\documentclass[12pt,oneside]{amsproc}
\usepackage[english]{babel}
\usepackage{a4wide}
\usepackage{amsthm}
\usepackage{graphics}
\usepackage{amsfonts, amssymb, amscd, amsmath}
\usepackage[matrix,arrow,curve]{xy}
\usepackage{tikz}
\usetikzlibrary{matrix,decorations.pathreplacing,positioning}
\usepackage[hidelinks]{hyperref}
\usepackage{comment}
\usepackage{tikz-cd}
\usepackage[justification=centering]{caption}

\DeclareMathOperator{\Ker}{Ker}

 \DeclareMathOperator{\im}{Im}

\newcommand{\Z}{\mathbb{Z}}
\newcommand{\R}{\mathbb{R}}

\newcommand{\RP}{\mathbb{R}P}

\newcommand{\dt}{\mathbb{Z}_2}

\newcounter{stmcounter}[section]

\newcounter{problcounter}

\numberwithin{equation}{section}

\theoremstyle{plain}
\newtheorem{cor}[stmcounter]{Corollary}

\newtheorem{thm}[stmcounter]{Theorem}
\newtheorem{thmNo}{Theorem}

\newtheorem{prop}[stmcounter]{Proposition}
\newtheorem{lem}[stmcounter]{Lemma}
\newtheorem{probl}[problcounter]{Problem}

\theoremstyle{definition}
\newtheorem{defin}[stmcounter]{Definition}

\theoremstyle{remark}
\newtheorem{ex}[stmcounter]{Example}
\newtheorem{rem}[stmcounter]{Remark}
\newtheorem{con}[stmcounter]{Construction}

\begin{document}
\title{Three-Dimensional Small Covers and Links}
\author{Vladimir Gorchakov}

\address{Department of Mathematics, University of Western Ontario}
\email{vyugorchakov@gmail.com}

\subjclass[2020]{ 57S12, 57S17, 57S25, 57M60,  52B10}

\begin{abstract}
We study certain orientation-preserving involutions on three-dimensional small covers. We prove that the quotient space of an orientable three-dimensional small cover by such an involution in  $\dt^3$ is homeomorphic to a connected sum of copies of $S^2 \times S^1$. If this quotient space is a 3-sphere, then the corresponding small cover is a two-fold branched covering of the 3-sphere along a link. We provide a description of this link in terms of the polytope and the characteristic function.

\end{abstract}

\maketitle

\section{Introduction}

In \cite{DJ}, M.\,W.\,Davis and T.\,Januszkiewicz introduced a class of $n$-dimensional closed manifolds called \emph{small
covers over simple $n$-polytopes}. 
A small cover can be defined using the following combinatorial data: a simple $n$-polytope $P$ and a map $\lambda$ from the set of facets $\mathcal{F} = \{F_1, \dots, F_m\}$ of $P$ to $\dt^n$ such that for any face $F$, that is an intersection of facets $F_i$, the corresponding vectors $\lambda_i = \lambda(F_i)$ are linearly independent. In particular, small covers have a natural $\dt^n$-action.  There is a close connection between topological and geometric properties of a small cover $X = X(P, \lambda)$ and combinatorial properties of a simple polytope $P$.

A closely related coloring construction arises in the theory of hyperbolic $3$-manifolds.
In the hyperbolic setting, A.\,Yu.\,Vesnin introduced in~\cite{Vesn1} a method for
constructing closed hyperbolic $3$-manifolds via $\dt^3$-colorings of facets of compact
right-angled hyperbolic polyhedra; see also~\cite{Vesn} for a historical overview.

In this paper, we focus on $3$-dimensional small covers. Note that by the Four Color Theorem every simple $3$-polytope admits a small cover. In \cite{BEMPP}, V.\,M.\,Buchstaber et al.  proved the cohomological rigidity  for $3$-dimensional small covers over Pogorelov polytopes.  In \cite{Er2}, N.\,Yu.\,Erokhovets constructed an explicit geometric decomposition of orientable $3$-dimensional small covers. In \cite{WuYu}, L.\,Wu and L.\,Yu obtained a criterion when a $3$-dimensional small cover is a Haken manifold. In \cite{Gruj}, V.\,Grujić obtained an explicit presentation of the fundamental group of an orientable $3$-dimensional small cover with the minimal number of generators. 

In a series of articles \cite{Med}, \cite{VesMed1}, \cite{VesnMed}, A.\,D.\,Mednykh and A.\,Yu.\,Vesnin studied $3$-manifolds with an involution such that the corresponding orbit space is a $3$-sphere. They called  such manifolds \emph{hyperelliptic}. In this paper, we study involutions on orientable three-dimensional small covers, in analogy with the hyperelliptic manifolds considered by Mednykh and Vesnin. Note that related work was recently  done independently by N.\,Yu.\,Erokhovets in \cite{Er}. In particular, some results of this paper are proved in \cite{Er}. However, our main results and approaches are different from \cite{Er}.

In our first main result, we describe the orbit space of a $3$-dimensional orientable small cover by an orientation-preserving involution $g \in \dt^3$.
\begin{thmNo} \label{Main1I}
    Let $X$ be an orientable $3$-dimensional  small cover, and let $g \in \dt^3$ be an orientation-preserving involution. Then the orbit space $X/g$ is homeomorphic to $\#_{k-1} S^2 \times~S^1$ for some $k$.
\end{thmNo}
The number $k$ will be defined later.
The idea is to study the following sequence of quotient maps: 
\begin{equation*} 
    X \to X/g  \to X/G \to X/\dt^3,
\end{equation*}
where $G$ is the orientation-preserving subgroup of $\dt^3$. In this case, the orbit space $X/G$ is homeomorphic to two copies of $P$ glued along the boundary and hence homeomorphic to $S^3$ as follows from \cite[Thm. 5.9]{Gor}. Hence, we can write this sequence as follows:
\begin{equation*} 
    X \to X/g  \to S^3 \to P.
\end{equation*}
It follows from \cite[Thm.\,1.7]{SmallCovers} that in the case of orientable $3$-dimensional small covers there are only two possibilities for the image of the characteristic function $\lambda$: $\im \lambda = \{ \lambda_1, \lambda_2, \lambda_3 \}$ or  $\im \lambda = \{ \lambda_1, \lambda_2, \lambda_3, \lambda_1 + \lambda_2 + \lambda_3 \}$. Moreover, in this case $G = \{0, \lambda_1 + \lambda_2, \lambda_1 + \lambda_3, \lambda_2 + \lambda_3\}$. Therefore, for every edge $I = F_i \cap F_j$, that is the intersection of facets $F_i$ and $F_j$, we can assign the element $\lambda(F_i) + \lambda(F_j) \in G$. This induces a labeling of edges of $P$ by elements of $G$, see Construction \ref{constr} for details.  In Theorem \ref{Main1}, we show  that the branching set of the $2$-fold branched covering $X/g \to S^3$ is the trivial link with $k$ components,  corresponding to all edges of $P$ that are \emph{not labeled} by $g$. Since a $2$-fold branched covering is determined by its branching link, we have that  $X/g \cong \#_{k-1} S^2 \times S^1$.

In particular, for $k = 1$,  the orbit space $X/g \cong S^3$ is a $2$-fold branched covering of the $3$-sphere $X/G \cong S^3$ along the trivial knot. This trivial knot corresponds to a Hamiltonian cycle in the simple polytope $P$, which we denote by $C$. Moreover, we show that there exists $g \in G$ such that $X/g \cong S^3$ if and only if the corresponding characteristic function $\lambda$ is induced by a Hamiltonian cycle. This result  was independently obtained by a different method in \cite{Er}. Note that this result provides a topological interpretation of $4$-colorings that are induced by  Hamiltonian cycles. In this case, we call $\lambda$ a \emph{Hamiltonian} characteristic function. The corresponding small covers are  hyperelliptic manifolds in the sense of \cite{Med} and \cite{Er}. 

A small cover $X = X(P,\lambda)$ with a Hamiltonian characteristic function $\lambda$ is a $2$-fold branched covering of $X/g \cong S^3$ along a link $L$. In this case, the Hamiltonian characteristic function determines an orientation-preserving involution $g \in \dt^3$. The link $L$ arises as the preimage under the map $X \to X/g$ of the edges of $P$ that are not contained in the Hamiltonian cycle $C$. Our second main result provides a combinatorial description of this link $L$ in terms of the pair $(P,\lambda)$. 

From a Hamiltonian cycle $C$ on a simple $3$-polytope $P$, the edges of $P$ that are not contained in $C$ are naturally separated into two disjoint subsets by the two sides of $C$ in $\partial P$. This yields a bipartite diagram $D_C$ associated with the Hamiltonian cycle $C$, whose vertex set is the set of vertices of $C$ equipped with the cyclic order induced by the cycle $C$, and whose arcs correspond to edges of $P$ not contained in $C$. A precise definition of the diagram $D_C$ is given in Section~5. 

To visualize the diagram $D_C$, we pass to a linear representation. From the Hamiltonian cycle $C$, we obtain a Hamiltonian path $H$ by removing an arbitrary edge from $C$, which induces a linear order on the set of vertices of $P$. We represent such diagrams graphically by placing the vertices on a line (representing the Hamiltonian path $H$) in the induced order and drawing each edge that is not in $C$ as an arc joining two vertices on one side of the line. Edges from the two subsets are drawn at different levels, so that edges from one subset pass over edges from the other. We obtain the link diagram $L(D_C)$ by doubling each chord of $D_C$ across the line, that is, by adding its image under a $\pi$-rotation about the axis. As a result, each edge gives rise to a trivial knot component, and the over/under information is reversed on the two sides. See the examples below.  
 
\begin{thmNo} \label{Main2I}
Let $X = X(P,\lambda)$ be an orientable $3$-dimensional small cover with Hamiltonian characteristic function $\lambda$.
Let $g \in \dt^3$ be the orientation-preserving involution determined by $\lambda$, and let $D_C$ be the bipartite  chord diagram associated with the Hamiltonian cycle $C$. Then $X \to X/g$ is a $2$-fold branched covering whose branching set is the link $L(D_C)$.
\end{thmNo}

\begin{ex}
    Let $X = \RP^3$ be the small cover over the simplex $\Delta^3$, and let $g = \lambda_2 + \lambda_3$.  Then the branching set of the map $X/g \to X/G$ is an unknot that corresponds to the Hamiltonian cycle in Figure~$1$, that is all the edges not labeled by $\lambda_2 + \lambda_3$. The orbit space $X/g$ is homeomorphic to $S^3$. The corresponding link $L$ is the Hopf link. 

\end{ex}

\begin{figure}[ht]
\begin{minipage}[ht]{0.35\linewidth}  
\center{\scalebox{1.5}{\includegraphics{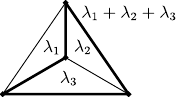}}}
\caption{The  Hamiltonian cycle in $\Delta^3$.}  
\end{minipage}
\hfill
\begin{minipage}[ht]{0.35\linewidth} 
\center{\scalebox{1.5}{\includegraphics{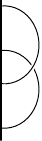}}}
\caption*{The bipartite  chord diagram $D_C$.} 
\end{minipage}
\hfill
\begin{minipage}[ht]{0.27\linewidth} 
\center{\scalebox{1.5}{\includegraphics{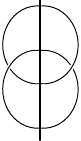}}}
\caption*{The Hopf link.} 
\end{minipage}
\end{figure}

\begin{ex}
Let $X = X(I^3, \lambda)$ be the small cover over cube $I^3$ with $\lambda$ as in Figures $2$ and $3$. For $g = \lambda_2 + \lambda_3$, the branching set of the map $X/g \to X/G$ is an unlink with two components that corresponds to two disjoint cycles in $I^3$ as shown in Figure~$2$. The orbit space $X/g$ is homeomorphic to $S^2 \times S^1$. 
\begin{figure}[ht]
\center{\scalebox{1.5}{\includegraphics{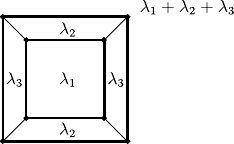}}}

\caption{Two cycles in $I^3$.}  

\end{figure}

For $g = \lambda_1 + \lambda_3$,  the branching set of the map $X/g \to X/G$ is an unknot that corresponds to the Hamiltonian cycle in Figure~$3$. The orbit space $X/g$ is homeomorphic to $S^3$. The corresponding bipartite  chord diagram and the link are shown below. The link $L$ is known as $L8n8$ in the Thistlethwaite link table. In particular, it is not alternating.

\end{ex}

\begin{figure}[ht]
\begin{minipage}[h]{0.35\linewidth}  
\center{\scalebox{1.5}{\includegraphics{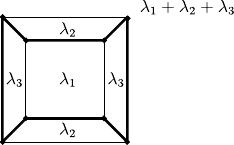}}}
\caption{The  Hamiltonian cycle in $I^3$.}  
\end{minipage}
\hfill
\begin{minipage}[h]{0.3\linewidth} 
\center{\scalebox{1.5}{\includegraphics{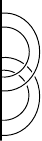}}}
\caption*{The bipartite  chord diagram $D_C$.} 
\end{minipage}
\hfill
\begin{minipage}[h]{0.28\linewidth} 
\center{\scalebox{1.5}{\includegraphics{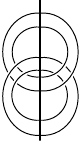}}}
\caption*{The link $L(D_C)$.} 
\end{minipage}
\end{figure}

\section{Preliminaries}

Now we recall some preliminaries about small covers. For further details see \cite{DJ}.
\begin{defin}
Let $P=P^n$ be a simple polytope of dimension $n$.
A \emph{small cover over} $P$ is a smooth manifold  $X=X^n$ with a locally standard smooth $\dt^n$-action such that the orbit space $X/\dt^n$ is homeomorphic to  $P$ as a manifold with corners. 
\end{defin}

 Let $\pi: X \to P$ be a small cover over $P$. For every face $F$ of $P$, denote its relative interior by  $F^{\circ}$.  For every $x,y \in~\pi^{-1}(F^{\circ})$ the stabilizer groups of $x$ and $y$ are the same. Denote this stabilizer group by $G_F$. In particular, if $F$ is a facet, then $G_F$  is a subgroup of rank one, hence $G_F = \langle \lambda(F)\rangle$ for $\lambda(F) \in \dt^n$. Hence, we get  a \textit{characteristic function} 
 $$\lambda : \mathcal{F} \to \dt^n,$$
 from the set $\mathcal{F}$ of all facets of $P$.
 We denote $\lambda(F_i)$ by $\lambda_i$.
For a codimension $k$ face $F$  we have that $F = F_1 \cap \dots \cap F_k$ for some facets $F_1, \dots, F_k \in \mathcal{F}$, then $G_F$ is a subgroup of rank equal to $k$ and generated by $\lambda_1, \dots, \lambda_k$. Therefore, we have the following $(*)$-condition
\\

$(*)$  Let $F = F_1 \cap \dots \cap  F_k$ be any codimension $k$ face of $P$. Then $\lambda_1, \dots, \lambda_k$ are linearly independent in $\dt^n$.
\\

Conversely, a simple polytope $P$ and a map $\lambda: \mathcal{F} \to \dt^n$  satisfying the $(*)$-condition determine a small cover $X(P,\lambda)$ over $P$ by the following construction:
\begin{equation} \label{DJCons}
    X(P, \lambda) = P \times \dt^n/\!\!\sim,
\end{equation}
where $(p,g) \sim (q, h)$ if and only if $p=q$ and $h-g \in G_{F(p)}$, where $F(p)$ is the unique face of $P$ that contains $p$ in its relative interior.  The $\dt^n$-action on $X(P, \lambda)$ is defined by the following formula:
$g (p,h) = (p,g+h)$.

\begin{thm}[{\cite[Prop.\,1.8]{DJ}}]\label{DJeqh}
    Let $X$ be a small cover over $P$ with characteristic function $\lambda: \mathcal{F} \to \dt^n$. Then $X$ and $X(P,\lambda)$ are $\dt^n$-equivariantly homeomorphic.
\end{thm}

In this article we work only with orientable small covers. We have the following criterion of orientability. 
\begin{thm}[{\cite[Thm.\,1.7]{SmallCovers}}] \label{orient}
    A small cover $X=X^n$ is orientable if and only if there exists  a linear functional $\xi \in (\dt^n)^*$ such that $\xi(\lambda_i) = 1$ for every facet $F_i \in \mathcal{F}$.
\end{thm}
On the other hand, every subgroup $G \subset \dt^n$ of rank $n-1$ is determined by a unique non-zero linear functional $\xi \in (\dt^n)^*$ by the correspondence $G = \Ker(\xi)$. In the case of an orientable small cover $X$ we have the following result. 
\begin{prop}
    Let $X = X(P, \lambda)$ be an orientable small cover,  let $\xi \in (\dt^n)^*$ be the linear functional such that $\xi(\lambda_i) = 1$ for every facet $F_i \in \mathcal{F}$. Then $G = \Ker (\xi: \dt^n \to \dt)$ is the orientation-preserving subgroup of $\dt^n$. 
\end{prop}
\begin{proof}
    It follows from the proof of {\cite[Thm.\,1.7]{SmallCovers}} that the combinatorial structure of $P$ defines a  cellular decomposition of $X = P \times \dt^n/\!\!\sim$ and the associated cellular chain complex $C_n$ in degree $n$ is the free abelian group generated by $\{P\} \times \dt^n = \{(P, h) : h \in \dt^n \}$. In the orientable case, a generator of $H_n(X, \Z)$ is $[X] = \sum_{h \in \dt^n} (-1)^{\xi(h)}(P,h)$. The $\dt^n$-action on generators of $C_n$ is given by $g \cdot (P,h) = (P, h+g)$ for any $g \in \dt^n$. For any $g \in G$ we have 
$$
g \cdot [X] = \sum_{h \in \dt^n} (-1)^{\xi(h)}(P,h+g)  = \sum_{h \in \dt^n} (-1)^{\xi(h)}(P,h) =[X],
$$
since $\xi(h+g) = \xi(h)$. On the other hand, if $\xi(g)=1$ then $g$ is an orientation-reversing element. Therefore, the subgroup $G = \Ker (\xi: \dt^n \to \dt)$ is the orientation-preserving subgroup of $\dt^n$. 

\end{proof}
\begin{rem}
    The subgroup $G$ is called \emph{$2$-subtorus in general position} in \cite{Gor}.
\end{rem}

The following lemma is crucial in proving Theorem \ref{OrbG} below. 
\begin{lem}\label{StGenPosSC}
    If $G \subset \dt^n$ is the orientation-preserving subgroup, then $G_F \not \subset G$ for every face in $\partial P$.
\end{lem}
\begin{proof}
    Indeed, $G_F$ is generated by $\lambda_i$ for some $i$. 
\end{proof}

\begin{rem}\label{GenPosSC}
    If we choose a vertex $p = F_{i_1} \cap \dots \cap F_{i_n}$, then the corresponding characteristic vectors $\lambda_{i_1}, \dots, \lambda_{i_n}$ are a basis of $\dt^n$. Then in this basis we have
    $$G = \{ (g_1, \dots, g_n) \in \dt^n : \sum_{i=1}^n g_i = 0 \}.$$

    In particular, for $n = 3$ we have 
    $$
    G = \{(0,0,0), (1,1,0), (1,0,1), (0,1,1)\}.
    $$
\end{rem}
\begin{rem}
    In this article we use only the additive notation, i.e. $\dt =~\{0, 1\}$.
\end{rem}

The following theorem was proved by the author in \cite[Thm. $5.9$]{Gor}. We provide a different proof here.  
\begin{thm} \label{OrbG}
    Let $X = X^n$ be an orientable small cover and $G$ be the orientation-preserving subgroup of $\dt^n$. Then the orbit space $X/G$ is homeomorphic to the $n$-dimensional sphere~$S^n$.
\end{thm}
\begin{proof}
We can assume that $X = X(P, \lambda) = P \times \dt^n/\!\!\sim$ by Theorem \ref{DJeqh}. Then we have the following sequence of homeomorphisms
$$
X/G \cong (P \times \dt^n/\!\!\sim)/G \cong P \times (\dt^n/G)/\!\!\sim,
$$
where $(p,g) \sim (q,h)$ in $P \times \dt^n/G$  if and only if $p=q$, and  $h=g$ in $\dt^n/G$ or $h-g \in G_{F(p)}$, where $F(p)$ is the unique face of $P$ that contains $p$ in its relative interior.

We claim that for any point $p \in \partial P$ we have $(p,g) \sim (p,h)$ for any two elements $g,h \in \dt^n/G$. 

Indeed, by Lemma \ref{StGenPosSC} there exists $t \in G_{F(p)}$ such that $t \neq 0$ in $\dt^n/G$, where $0$ is the identity element of $\dt^n/G$. Hence, for any $h \in \dt^n$ we have two cases: If $h \notin G$, then $(p,h) \sim (p,t)$, since $t-h \in G$; otherwise $(p,h) \sim (p,0)$. On the other hand, $(p,t) \sim (p,1)$, since $t \in G_{F(p)}$. Hence, for any $h \in \dt^n$ it holds that $(p,h) \sim (p,0)$.

On the other hand, if $p$ in the interior of $P$, then $G_F(p) = \{0\}$. Therefore, we have
$$
X/G \cong P \times \dt/\!\!\sim,
$$
where $(p, 0) \sim (p, 1)$ if and only if $p \in \partial P$. Since $P$ is homeomorphic to $D^n$, we get that $X/G \cong S^n$.
\end{proof}

\begin{cor} \label{Gluing}
    The map $X/G \to P^n$ is the quotient map of the involution that swaps the two hemispheres with $\partial P^n = S^{n-1}$ as the fixed point set.
\end{cor}
\begin{proof}
Indeed, we have that 
    $$
    X/G \cong P \times \dt/\!\!\sim,
    $$
where $(p, 0) \sim (p, 1)$ if and only if $p \in \partial P$. Then the $\dt^n/G$-action on $P \times \dt/\!\!\sim$ swaps two copies of $P$ with $\partial P$ as the fixed point set. 
\end{proof}
\subsection{Preliminaries on 2-fold branched coverings}
Now we state some facts from the theory of $3$-manifolds. All manifolds, their submanifolds and maps are PL. We refer to \cite{Rolf}, \cite{Viro} for details. 
\begin{defin}
    Let $M$, $N$ be two orientable $3$-manifolds and $L \subset N$ be a locally flat $1$-dimensional submanifold.  A  map $\pi: M \to N$ is called a \textit{$2$-fold branched covering of $N$ with a branching along $L$}, if $\pi_{|M\setminus \pi^{-1}(L)}: M\setminus \pi^{-1}(L) \to N\setminus L$ is a covering map of degree $2$ and $\pi_{|\pi^{-1}(L)}: \pi^{-1}(L) \to L$ is a homeomorphism. 
\end{defin}

We say that two $2$-fold branched coverings $\pi: M \to N$, $\pi^{\prime}: M^{\prime} \to N$ are \textit{equivalent} if there exists a PL homeomorphism $f: M \to M^{\prime}$ such that $\pi = \pi^{\prime} \circ f$. If $\pi: M \to N$ is a $2$-fold branched covering, then the only non-trivial automorphism of this covering $\tau : M \to M$ is an involution and $M/ \tau \cong N$. Thus every $2$-fold branched covering is equivalent to a quotient map by an involution. In these terms, equivalence of $2$-fold branched coverings is  a $\dt$-equivariant PL homeomorphism. 
For the details see \cite[Section 1.3]{Viro}. 

Let $L$ be a link in $S^3$. Then the following holds.

\begin{thm}[\cite{Viro}, Section 1.4] \label{unq}
    There exists a unique, up to equivalence, $2$-fold branched covering $\pi: M \to S^3$ with a branching along $L$.
\end{thm}

Note that there exist non-equivalent links with homeomorphic $2$-fold branched coverings, see \cite[Section 3.7]{Viro}. However, in the case of the trivial link we have the following.

\begin{thm}[\cite{Unlink}] \label{trivial}
    The trivial link is determined by its $2$-fold branched covering space. More precisely,  $\#_{k-1} S^2 \times S^1$ arises as the $2$-fold  branched cover of only the trivial link of $k$ components.
\end{thm}
For the proof see, for example, the last Corollary in \cite{Unlink}. In the case of the unknot this statement is a particular case of the Smith Conjecture for involutions and it was proved by F.\,Waldhausen, \cite{Wald}.

\section{Small covers as 2-fold branched coverings}

Let $X = X(P, \lambda)$ be an orientable $3$-dimensional  small cover over a simple polytope $P$. For every facet $F_i$ in $P$ let $\lambda_i = \lambda(F_i)$, where $\lambda$ is the characteristic function. Without loss of generality, we can assume that  $\lambda_1, \lambda_2, \lambda_3$ are basis vectors of $\dt^3$. 
The next proposition easily follows from Theorem \ref{orient}.
\begin{prop} \label{colouring}
    For a characteristic function $\lambda$ we have only two possibilities for the image of~$\lambda$:
    $\im \lambda = \{ \lambda_1, \lambda_2, \lambda_3 \}$ or 
    $\im \lambda = \{ \lambda_1, \lambda_2, \lambda_3, \lambda_1 + \lambda_2 + \lambda_3 \}$.
\end{prop}
\begin{proof}
    Indeed, by Theorem \ref{orient} there exists a linear functional $\xi \in (\dt^3)^*$ such that $\xi(\lambda_i) = 1$ for all $i$. Hence, there are 2 cases: $\lambda_1 + \lambda_2 + \lambda_3 \in \im\lambda$ or $\lambda_1 + \lambda_2 + \lambda_3 \notin \im\lambda$. This proves the statement.  
\end{proof}

Let $G$ be the orientation-preserving subgroup of $\dt^3$.  Then for any $g \neq 0$  in $G$ we have the following sequence of the quotient maps:
\begin{equation} \label{seq}
    X \to X/g  \to X/G \to X/\dt^3.
\end{equation}

Since $X/G \cong S^3$ and $X/ \dt^3 \cong P$, we can rewrite it as follows: 
$$
X \to X/g  \to S^3 \to P,
$$
where $X/g \to S^3$ is the quotient map of the $G/g$-action on $X/g$.  
In the next lemma, we show that we can work in the PL category.

\begin{lem}
The maps $X \to X/g$ and $X/g \to S^3$ are $2$-fold branched coverings.
\end{lem}

\begin{proof}
We first show that $X/g$ is a topological manifold.  Let $x \in X$ be a fixed point, then it follows from the Slice Theorem that there exists a  $g$-stable neighborhood $U$ of $x$ such that $U$ is $g$-equivariantly homeomorphic to the tangent space $T_{x}X$ at the point $x$ with the tangent representation of $g$. Since $g$ is orientation-preserving it follows that $U/g \cong \R^3$. If $x$ lies in a free orbit, then there exists an open neighborhood $U$ of the orbit of $x$ that is $g$-equivariantly homeomorphic to $\dt \times \R^3$, hence $U/g \cong \R^3$. 

Since the $\dt^3$-action on $X$ is smooth there exists  a $\dt^3$-equivariant triangulation of $X$, see \cite{EqTrian}. This implies that the maps $X \to X/g$ and $X/g \to S^3$ are simplicial. On the other hand, every triangulation of a $3$-manifold is a combinatorial triangulation, see \cite[Theorem 1]{PL}, and hence $X$ and $X/g$ are PL manifolds. By the same local analysis as above it follows that the maps are $2$-fold branched coverings. 
\end{proof}

\begin{rem}
    In the previous lemma we proved that $X/g$ is a topological manifold. A more general statement was proved in \cite[Corollary 1.17]{Er}. For a more general statement in the torus case see \cite{AyzGor}.
\end{rem}

 To describe the branching sets of every map in Sequence \eqref{seq} we introduce the following construction. 
\begin{con} \label{constr}

Let $p: X \to P$ be the quotient map. For every face $F$ in $P$ we can define a small cover $X_F = p^{-1}(F)$ over $F$ with the $\dt^3/G_F $-action. In particular, for edges we have that $X_F$ is equivariantly homeomorphic to $S^1$ with the $\dt$-action with two fixed points. 

Let $\Gamma$ be the $1$-skeleton of $P$. Then we have that $\Gamma' = p^{-1}(\Gamma)$ is a graph in $X$ with fixed points as vertices and circles as edges. It follows from Proposition \ref{colouring} that for every edge $I = F_i \cap F_j$, that is the intersection of facets $F_i$ and $F_j$, we can assign the element $\lambda_i + \lambda_j \in G$ such that the corresponding small cover $X_I \cong S^1$ is fixed by this element.
\end{con}
Now we can prove Theorem \ref{Main1I} from the introduction. 
 
\begin{thm} \label{Main1}
    For any non-identity $g \in G$  the branching set $L$ in $S^3$ of the map $X/g \to S^3$  is the trivial link, i.e. disjoint union of unknots. Hence $X/g \cong \#_{k-1} S^2 \times S^1$, where $k$ is the number of connected components of $L$.
\end{thm}
\begin{proof}
We have the following commutative diagram:
\[
\begin{tikzcd}
    X \arrow{r}{\varphi_1} & X/g \arrow{r}{\varphi_2} & X/G \arrow{r}{\varphi_3} & P \\
    \Gamma' \arrow{r} \arrow[hook']{u} & \Gamma'/g \arrow{r} \arrow[hook']{u} & \Gamma \arrow[hook']{u} \arrow{r}{id} &\Gamma \arrow[hook']{u}
\end{tikzcd}
\]
where $\Gamma'/g$ has $S^1$-edges for edges labeled with $g$ and $I^1$-edges otherwise, since $S^1/g \cong I^1$ for these edges. Note that $\Gamma' /G = \Gamma$ in $X/G$, i.e. every edge in $\Gamma/G$ is an $I^1$-edge.  

The last map at the bottom is the identity map as follows from Corollary \ref{Gluing}. Indeed,  we get $S^3$ as two copies of $P$ glued by the boundary. Since $\Gamma \subset \partial P$ we get that the branching set of the map $X/g \xrightarrow{\varphi_2} X/G$ lies in the boundary of $P$, hence it is the trivial link.  Hence, $X/g \to S^3$ is the double branched covering of $S^3$ over a disjoint union of unknots.  In the case of a disjoint union of unknots, we have that $X/g \cong \#_{k-1} S^2 \times S^1$, where $k$ is the number of connected components of $L$ as follows from Theorem \ref{trivial}.
\end{proof} 
\begin{rem} \label{rem1}
    It follows from the proof of Theorem \ref{Main1} that these unknots are  disjoint cycles on $P$ such that every vertex of $P$ lies in some of these cycles.  In particular, for the case $k = 1$ we get a \textit{Hamiltonian} cycle on $P$.  See Section $\ref{SphereCase}$  for the further discussion. 
\end{rem}
\begin{cor} \label{MQ}
    Let $X$ be a $3$-dimensional orientable small cover. Then $X$ is a $2$-fold branched covering over the sphere with $k-1$ handles, i.e. $\#_{k-1} S^2 \times S^1$.
\end{cor}

 The following question was asked in \cite[Question 2]{MonConj}:

   \textit{ Is every closed,  orientable  $3$-manifold  a  $2$-fold  covering branched over a  $3$-sphere  with handles?}

For small covers, Corollary \ref{MQ} provides an affirmative answer to this question.
However, in general the answer is negative, since there exist $3$-manifolds on which every finite group action is trivial; see \cite{NoAct}. From Corollary \ref{MQ} we also obtain the following.

\begin{cor}[\cite{MonConj}] \label{HamiltonCycles}
     Let $X$ be an orientable $3$-dimensional  small cover such that  $H^{1}(X, \mathbb{Q})=0$. Then $X/g \cong S^3$ for any non-identity $g \in G$.
\end{cor}
\begin{proof}
   Assume the converse. Then there exists $g \in G$ such that $X/g \cong \#_{k-1} S^2 \times S^1$ for $k > 1$. By the transfer theorem, see \cite[Chap. 3, Thm. 7.2]{Bredon}, we have 
    $$H^1(X/g, \mathbb{Q}) \cong H^{1}(X, \mathbb{Q})^{g}.$$ Hence, $H^1(X, \mathbb{Q}) \neq 0$ and we get a contradiction. 
\end{proof}

\begin{rem}
   Integral cohomology groups of small covers were calculated in \cite{CaiCho21} by L.\,Cai and S.\,Choi. In particular, $3$-dimensional small covers have at most 2-torsion in the cohomology groups. Note that an algorithm for detecting rational homology spheres among small covers was obtained in  \cite[Cor. 7.9]{RatSph}. 
\end{rem}
N.\,Yu.\,Erokhovets independently proved a more general statement.
\begin{thm}[{\cite[Prop. 12.13]{Er}}] \label{ErokhThm2}
    Let $X$ be an orientable $3$-dimensional small cover. Then $H^1(X, \mathbb{Q}) = 0$ if and only if $X/g \cong S^3$ for any non-identity $g \in G$.
\end{thm}
Moreover, small covers as in Theorem \ref{ErokhThm2} correspond to triples of Hamiltonian cycles on a simple polytope $P$ such that any edge of $P$ belongs to exactly two cycles, see \cite[Thm. 11.7]{Er}. For further details see Section 13 in \cite{Er}.

\section{Small covers as 2-fold branched coverings of the sphere} \label{SphereCase}
In this section, we study orientable $3$-dimensional small covers $X = X(P,\lambda)$ for which the quotient $X/g$ is homeomorphic to the $3$-sphere $S^3$ for some non-identity element $g$ of the orientation-preserving subgroup $G \subset \dt^3$.

As shown in the previous section (Theorem~\ref{Main1}), the condition $X/g \cong S^3$ implies that the branching set of the map $X/g \to X/G$ is an unknot. This unknot determines a Hamiltonian cycle on the polytope $P$, which gives rise to a combinatorial description of the condition $X/g \cong S^3$,  which is made precise in Proposition~\ref{ErokhTh}.

We call a characteristic function $\lambda$ a \emph{regular $4$-coloring} if
\[
\im\lambda=\{\lambda_1,\lambda_2,\lambda_3,\lambda_1+\lambda_2+\lambda_3\}.
\]
By Theorem~\ref{orient}, any regular $4$-coloring defines an orientable small cover.
The next proposition is well known. 
\begin{prop}
Any Hamiltonian cycle in a $3$-polytope $P$ induces the regular $4$-coloring.
\end{prop}
\begin{proof}
Indeed, a Hamiltonian cycle separates $\partial P$ into two parts. It can be shown, for example see Construction $11.2$ in \cite{Er}, that the facets in each part can be colored in two colors such that adjacent facets have different colors. This gives us the required characteristic function $\lambda$.
\end{proof}
\begin{rem}
    There exist $4$-colorings that do not arise from a Hamiltonian cycle. For example, see \cite{Grinb}.
\end{rem}

Let $\lambda$ be a characteristic function, which arises from a Hamiltonian cycle.  In this case, we say that $\lambda$ is \textit{Hamiltonian}. Note that different cycles can give the same $\lambda$. For example, this is the case of Theorem \ref{ErokhThm2}. 

One direction of the next proposition was proved by A.\,Yu.\,Vesnin, A.\,D.\,Mednykh in \cite{VesnMed}, also see \cite[Theorem 4.1]{Vesn}. They proved this for bounded right-angled polytopes in $\mathbb{H}^3$, however, the same argument works for any simple $3$-polytope. The same result was obtained by N.\,Yu.\,Erokhovets by a different method in \cite[Theorem 11.5]{Er}.
\begin{prop} \label{ErokhTh}
        Let $X = X(P, \lambda)$ be an orientable $3$-dimensional small cover, let $G$ be the orientation-preserving subgroup of $\dt^3$. Then $\lambda$ is Hamiltonian if and only if $X/g \cong S^3$ for some non-identity $g \in G$.
\end{prop}
\begin{proof}
    Let $\lambda$ be a Hamiltonian characteristic function on $P$ and $X = X(P, \lambda)$ be the corresponding small cover. The corresponding Hamiltonian cycle divides $\partial P$ into two parts. Without loss of generality, we can assume that $\lambda_1, \lambda_2$ are the characteristic vectors for the facets in one part, and $\lambda_3,  \lambda_1 + \lambda_2 + \lambda_3$ are the characteristic vectors for the facets in another part. Let $g = \lambda_1 + \lambda_2$. We claim that $X/g \cong S^3$. Indeed, the Hamiltonian cycle corresponds to the edges which are \emph{not} labeled by $g$, see Construction  \ref{constr}. Hence this Hamiltonian cycle corresponds to the unknot in the branching set of the map $X/g \to X/G$, see the proof of Theorem \ref{Main1}. Hence $X/g \cong S^3$. The other direction of the statement follows from Theorem \ref{Main1}, see Remark \ref{rem1}. 
\end{proof}
\begin{rem}
    The involution $g$ from Proposition \ref{ErokhTh} is an example of a  \emph{hyperelliptic} involution in the sense \cite{VesnMed}, \cite{Er}.
\end{rem}
From Corollary \ref{HamiltonCycles} we get the following application to Hamiltonian cycles and regular 4-coloring of polytopes. 
\begin{cor}
    Let $\lambda$ be a regular $4$-coloring of a simple polytope $P$. Let $X = X(P, \lambda)$ be the corresponding orientable small cover. If $H^1(X, \mathbb{Q}) = 0$, i.e. if $X$ is a rational homology $3$-sphere, then the $4$-coloring arises from a Hamiltonian cycle. 
\end{cor}

For a topological space $X$ let $L(X)$ be the maximal integer $m$ such that $a_1a_2 \dots a_m \neq~0$, where $a_i \in H^{j}(X, \mathbb{Q})$ for $j \geq 1$. Then $L(X)$ is called \emph{rational cup-length} of $X$. 
\begin{cor} \label{length}
    Let $\lambda$ be a regular $4$-coloring of a simple polytope $P$ that arises from a Hamiltonian cycle.  Let $X = X(P, \lambda)$ be the corresponding orientable small cover. Then  $L(X) \leq 2$.  
\end{cor}
\begin{proof}
    Indeed, for a $2$-fold branched covering $X \to Y$, we have $L(X) \leq 2L(Y)$ by Theorem $2.5$ in \cite{degcov}. For $Y = S^3$, we have $L(S^3)=1$, hence $L(X) \leq 2$.  
\end{proof}
Note that for an orientable manifold $X$, we have $L(X) \leq \dim X$.
Therefore, the inequality in Corollary \ref{length} can be reformulated as $L(X) \neq 3$.

\begin{rem}
The rational cohomology ring of a small cover $X$ was calculated in \cite{RatCoh} by S.\,Choi and H.\,Park. This provides us with a way to calculate the rational cup-length and, consequently, an approach to verify the necessary condition stated in Corollary \ref{length} for a regular $4$-coloring to arise from a Hamiltonian cycle on $P$.
\end{rem}
\section{The link description}
Let $X = X(P, \lambda)$ be an orientable $3$-dimensional small cover with Hamiltonian characteristic function $\lambda$. Then there exists $g \in G$ such that $X/g \cong S^3$ and the quotient map $X \to X/g$ is a double branched covering of the sphere branched along some link $L$.
 In this section, we describe the link $L$  from the combinatorial data  $(P, \lambda)$.  The results of this section are motivated by an example due to A.\,Yu.\,Vesnin and A.\,D.\,Mednykh \cite[Ex.~3.1]{VesnMed}.

We now introduce the main combinatorial construction used to describe the link $L$.

\begin{defin}
Let $P$ be a simple $3$-polytope with a Hamiltonian cycle $C$.
The \emph{bipartite chord diagram associated with $C$}, denoted by $D_C$, consists of the following data:
\begin{itemize}
\item the vertex set $V(C)$ equipped with the cyclic order induced by the cycle $C$;
\item for each edge of $P$ not contained in $C$, an arc connecting the corresponding pair of vertices of $C$, called a \emph{chord};
\item a partition of the set of chords into two disjoint subsets, induced by the two sides of the cycle $C$ in $\partial P$.
\end{itemize}
\end{defin}

For graphical purposes, we represent the diagram $D_C$ as follows.
Removing an arbitrary edge from the Hamiltonian cycle $C$ turns it into a Hamiltonian path
\[
H = \{\{1,2\}, \{2,3\}, \dots, \{2l-1,2l\}\},
\]
which induces a linear order on the vertex set $V(C)$.
The choice of the removed edge is inessential and serves only to fix a linear order;
different choices lead to different linear representations of the same bipartite chord diagram $D_C$ and to different diagrams of the same link.

Using such a linear representation, we place the vertices of $V(C)$ on a line, representing the Hamiltonian path $H$, in the given order, and draw each chord $(i,j)$ as an arc joining $i$ and $j$ on one side of the line. Chords from the two subsets are drawn at different levels, so that chords from one subset pass over chords from the other. 
We obtain the link diagram $L(D_C)$ by doubling each chord of $D_C$ across the line, that is, by adding its image under a $\pi$-rotation about the axis.
 As a result, each chord gives rise to a trivial knot component, and the over/under information is reversed on the two sides.  We illustrate this construction by an example. 

\begin{ex}
The figure below shows a Hamiltonian cycle on the $5$-prism $P$
together with the associated bipartite chord diagram $D_C$ and the corresponding link $L(D_C)$.
By Proposition~\ref{ErokhTh}, there exists a small cover
$X = X(P,\lambda)$ which is a $2$-fold branched covering
$X \to X/g$ with branching set $L$.
Theorem~\ref{Main2} implies that $L(D_C)$ is a diagram of the link $L$.

\begin{figure}[ht]
\centering

\begin{minipage}{0.35\linewidth}
\centering
\scalebox{1.7}{\includegraphics{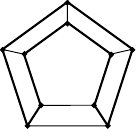}}
\caption*{The Hamiltonian cycle in $P$.}
\end{minipage}\hfill
\begin{minipage}{0.35\linewidth}
\centering
\scalebox{1.7}{\includegraphics{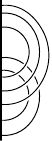}}
\caption*{The bipartite  chord diagram $D_C$.}
\end{minipage}\hfill
\begin{minipage}{0.28\linewidth}
\centering
\scalebox{1.7}{\includegraphics{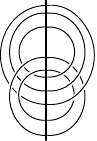}}
\caption*{The link $L(D_C)$.}
\end{minipage}

\end{figure}
\end{ex}

\begin{defin}
Two chords $(i_1,j_1)$ and $(i_2,j_2)$  are said to \emph{intersect} if exactly one of the vertices of chord $(i_2,j_2)$ lies between $i_1$ and $j_1$, that is, if
\[
i_1 < i_2 < j_1 < j_2 \quad \text{or} \quad i_2 < i_1 < j_2 < j_1.
\]
\end{defin}

Recall that an \emph{$n$-bridge decomposition} of a link $L$ is a decomposition $S^3 = B_1 \cup B_2$ into two $3$-balls such that the sphere
$E = B_1 \cap B_2$, called the \emph{bridge sphere}, intersects $L$ transversely and $L \cap B_i$ is a collection of $n$ trivial arcs for $i \in \{1,2\}$. The minimal such $n$ is called the \emph{bridge index} of $L$.
In the proof below, we construct an $l$-bridge decomposition of the link $L$.

\begin{thm} \label{Main2}
    The diagram $L(D_C)$ is a link diagram of the link $L$. 
\end{thm}
\begin{proof}
    Recall that $X/g \xrightarrow{\varphi_2} X/G$ is the double branched covering over the trivial knot, that is, the Hamiltonian cycle $C \subset \partial P$ in $X/G \cong  P \cup_{\partial P} P \cong S^3$. The link $L \subset X/g$ is the preimage, under the map $\varphi_2$, of all edges of $P$ that are not contained in $C$.  All these edges lie in the boundary of $P$. Moreover, the Hamiltonian cycle $C$  induces a partition of the edges into two classes, which we denote by $A$ and $B$.

Thus, topologically we have the $2$-sphere $S^2 \cong \partial P$ with a circle, that is, the Hamiltonian cycle $C$, which separates $S^2$ into two hemispheres $S^2 \cong D^2_A \cup_{C} D^2_B$ and two sets of arcs $ A = \{ a_1, \dots, a_p \} \subset D^2_A$ and $B = \{b_1, \dots, b_q\} \subset D^2_B$ with endpoints in $C$. The preimages of these hemispheres under the map $\varphi_2$ are spheres, which we denote by $S^2_A = \varphi^{-1}_2(D^2_A)$ and $S^2_B = \varphi^{-1}_2(D^2_B)$. Therefore, all components of $L$, which are preimages of the arcs, are trivial knots, since they lie in $2$-spheres. The intersection of $S^2_A$ and $S^2_B$ is the circle $\varphi_2^{-1}(C)$ and the order of endpoints of arcs, induced by $D_C$ and corresponding to vertices of $P$, does not change on the circle $\varphi_2^{-1}(C)$, since the restriction of $\varphi_{2}$ to  $\varphi_2^{-1}(C)$ is an orientation-preserving homeomorphism. Thus, we have $p$ trivial knot components in $S^2_A$ and $q$ trivial knot components in $S^2_B$.  The bipartite  chord diagram $D_C$ encodes the linking pattern of $L$: two components of $L$ are linked if and only if the corresponding chords in $D_C$ intersect.  Moreover, they are linked in the same way as in the Hopf link. 

Now consider a disk $D^2$ with boundary $\partial D^2 = C$ such that $D^2$ intersects the arcs only in $C$. The preimage $E=\varphi^{-1}(D^2)$ is a bridge sphere for the link $L$, which gives an $l$-bridge decomposition of  $L$, and after a small isotopy of $L$ the projection onto $E$ yields the diagram $L(D_C)$.
\end{proof}

\begin{rem}
The links constructed above are \emph{strongly invertible}, that is, there exists an orientation-preserving involution $h \in G/g$ such that
$h(K_i)=K_i$ and the restriction $h|_{K_i}$ has exactly two fixed points for each knot component $K_i \subset L$. In the proof above, the link $L$ is described via its quotient under this involution, which is encoded by the bipartite  chord diagram $D_C$. For strongly invertible knots, it is known that there is a bijective correspondence between knots equipped with such involutions and their quotients; see \cite{BBHL}. See also \cite[Sec.~3]{CMB}.
\end{rem}

\begin{cor}
    The bridge index of the link $L$ is $l$, where $2l$ is the number of  vertices of $P$. 
\end{cor}
\begin{proof}
    In the proof above we constructed an $l$-bridge decomposition of the link $L$. On the other hand, $L$ has $l$ components, and therefore the bridge index of $L$ is bounded below by $l$. 
\end{proof}
\begin{rem}
Further relations between these links and ideal right-angled $3$-polytopes were studied in a recent work of N.\,Yu.\,Erokhovets~\cite{Er3}.
\end{rem}

\section{Remarks and further directions}
In this section, we collect several remarks and problems related to the link
description constructed above, with an emphasis on rigidity questions. 
We begin with a remark related to a rigidity problem in toric topology.
\begin{rem} \label{BirmanProblem}
The following problem was posed by J.\,Birman in \cite[Problem 3.25]{Kirby}. Let $L$ be a link and $M_2(L)$ its $2$-fold branched covering. Let us say that $L_1$ and $L_2$ are \emph{$2$-equivalent} if and only if the corresponding manifolds $M_2(L_1)$ and $M_2(L_2)$ are homeomorphic. What is the description of a class of $2$-equivalent links? 
\end{rem}
Now let us recall some definitions from the theory of simple polytopes. For a detailed exposition we refer to \cite{BEMPP}. Recall that  a simple polytope $P$ is called a \emph{flag polytope} if every collection of pairwise intersecting facets of it has a non-empty intersection.  A \emph{$k$-belt} in a simple $3$-polytope is a cyclic sequence
$B_k = (F_{i_1}, \dots, F_{i_k})$ of $k > 3$ facets in which pairs of consecutive facets (including $F_{i_k}$ and $F_{i_1}$) are adjacent, other pairs of facets are disjoint, and no three facets have
a common vertex.

In the case of small covers with Hamiltonian characteristic functions we have the following result related to Remark \ref{BirmanProblem}.
\begin{prop}
     Let $P_1$ be a Pogorelov polytope, i.e. it is flag and does not have $4$-belts. Let $X_1 = X(P_1, \lambda_1)$, $X_2 = X(P_2, \lambda_2)$ be the small covers with Hamiltonian characteristic functions $\lambda_1, \lambda_2$. Let $L_1, L_2$ be the corresponding links. Then $L_1$ and $L_2$ are $2$-equivalent if and only if   $L_1$ and $L_2$ are equivalent. 
\end{prop}
\begin{proof}
    Indeed, if the two links are equivalent, then they are $2$-equivalent by Theorem \ref{unq}. On the other hand, $X_1$ is homeomorphic to $X_2$ if and only if the characteristic pairs $(P_1, \lambda_1)$ and $(P_2, \lambda_2)$ are equivalent, as follows from \cite[Thm. 5.6]{BEMPP}. However, the link $L$ is completely determined by the pair $(P, \lambda)$, hence $L_1$ is equivalent to $L_2$. 
\end{proof}
\begin{rem}
    Note that if $P$ is a Pogorelov polytope, then the small cover $X(P,\lambda)$ is a hyperbolic $3$-manifold of L\"{o}bell type, see \cite[Thm. 2.15]{BEMPP}.
\end{rem}
 
For general polytopes, however, the previous proposition is not true, as illustrated by the following example.

\begin{ex} \label{trunsimpl}
  Let $P_i$, where $i = 1, 2, 3$, be a polytope obtained from  a threefold vertex truncation (\cite[Construction 1.1.12]{BP}) of the simplex $\Delta^3$. Let $X_i = X(P_i, \lambda_i)$ be the corresponding small cover over $P_i$, see the appendix for the precise definition of the pairs $(P_i, \lambda_i)$. By Theorem \ref{Main1}, we have $X_i/g \cong S^3$ for any nontrivial $g \in G$.
By Theorem \ref{Main2}, the branching sets of the coverings
$X_i \to X_i/g$ are described, in terms of the links shown below, as follows:

\begin{enumerate}
    \item For $X_1$, the branching set is the same for all nontrivial $g \in G$,
    namely, the link $L_1$.

    \item For $X_2$, the branching set corresponding to $g=\lambda_1+\lambda_2$
    is the link $L_2$, while for all other $g \in G$ the branching set is the link $L_3$.

    \item For $X_3$, the branching set corresponding to $g=\lambda_1+\lambda_3$
    is the link $L_3$, while for all other $g \in G$ the branching set is the link $L_1$.
\end{enumerate}

  Thus, for example, for $X_2$ we have two different, but $2$-equivalent links. All $X_i$ are $\dt$-equivariantly homeomorphic by Theorem \ref{unq}.
\begin{figure}[h]
\begin{minipage}[h]{0.32\linewidth}  
\center{\scalebox{1.7}{\includegraphics{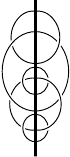}}}
\caption*{The link $L_1$}  
\end{minipage}
\hfill
\begin{minipage}[h]{0.32\linewidth} 
\center{\scalebox{1.7}{\includegraphics{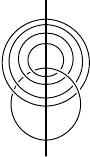}}}
\caption*{The link $L_2$} 
\end{minipage}
\hfill
\begin{minipage}[h]{0.32\linewidth} 
\center{\scalebox{1.7}{\includegraphics{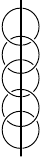}}}
\caption*{The link $L_3$} 
\end{minipage}
\end{figure}

\end{ex}

More generally, by Theorem~\ref{ErokhThm2}, any orientable $3$-dimensional small cover $X$ that is a rational homology sphere admits three involutions $g_i \in G$, $i\in\{1,2,3\}$. Each involution $g_i$ yields a link $L_i \subset X/g_i \cong S^3$. These links are not necessarily pairwise distinct.

\begin{probl}
Do these links $L_i$ determine a small cover $X$ up to a $\dt^3$-equivariant homeomorphism? More generally, what can be said about rigidity problems in the case of a Hamiltonian characteristic function?  
\end{probl}
Let $X_i = X_i(P_i, \lambda_i)$, where $i \in \{1, 2\}$,  be small covers  with Hamiltonian characteristic functions. Let $g_i$ be an involution such that $X_i/g_i \cong S^3$ and let $\dt = <g_i>$. By Theorem \ref{DJeqh} and Theorem \ref{unq} the small covers $X_i$ are
\begin{itemize}
    \item $\dt^3$-equivariantly homeomorphic  if and only if the pairs $(P_i, \lambda_i)$ are equivalent. 
    \item $\dt$-equivariantly homeomorphic if and only if the links $L_i$ are equivalent. 
\end{itemize}
In particular, link invariants give us invariants for $\dt$-equivariant homeomorphisms. For example, in the flag case we have the following invariant.

\begin{rem}
To a link $L$ we can associate a $\pi$-orbifold group $O(L)$, see \cite{orb}. 
A link $L$ is \emph{sufficiently complicated}, if the $2$-fold branched covering $M_2(L)$ is an aspherical space, see \cite[Prop. 1.1]{orb} for another definition.  On the other hand, a small cover $X(P, \lambda)$ is an aspherical space if and only if $P$ is a flag polytope, see \cite[Thm. 2.2.5]{DJS}. It was proved in \cite[Thm. 1]{orb} that two sufficiently complicated links are equivalent if and only if their $\pi$-orbifold groups are isomorphic.   
\end{rem}
Recall that an orientable $3$-manifold (or a link) is called \emph{prime} if it cannot be decomposed into a nontrivial connected sum of two manifolds (links). Every orientable $3$-manifold can be uniquely decomposed into the connected sum of prime manifolds. The same applies to links. The explicit prime decomposition of small covers was obtained in \cite{Er2}. The following problem arose from discussions with N.\,Yu.\,Erokhovets.
\begin{probl}
    Let $X = X(P,\lambda)$ be a small cover with Hamiltonian characteristic function. It follows from  Proposition \ref{ErokhTh} that $X$ is a $2$-fold branched covering of a link $L$, described in Section~5. What is the explicit prime decomposition of $L$? 
\end{probl}
In particular, we have the following result, also see \cite{dec}.
\begin{prop}[\cite{Unlink}]
    Suppose that $X$ is a $2$-fold branched covering of $S^3$ along a link $L$, and that $X$ contains no $S^1 \times S^2$ summands. If $X$ splits as connected sum $X = \#^n_{i = 1 } X_i$, there is a corresponding splitting of $L$ as $\#^n_{i = 1 } L_i$ with $X_i$ the $2$-fold branched covering of $L_i$.
\end{prop}

\section*{Acknowledgements}
The author is grateful to the anonymous referee for helpful comments that significantly improved the paper. The author would like to thank his advisor Matthias Franz for his attention to this work. The author also thanks N.\,Bogachev, N.\,Yu.\,Erokhovets, V.\,Shastin, and F.\,Vylegzhanin for useful discussions and comments. 
\clearpage
\appendix
\section*{Appendix}
In this section we define pairs $(P_i, \lambda_i)$, where $i \in \{1, 2, 3\}$, for Example \ref{trunsimpl}. See figures below. 
\begin{figure}[h]
\begin{minipage}[h]{0.49\linewidth}  
\center{\scalebox{1.5}{\includegraphics{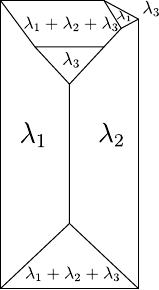}}}
\caption*{The  pair $(P_1, \lambda_1)$}  
\end{minipage}
\hfill
\begin{minipage}[h]{0.49\linewidth} 
\center{\scalebox{1.5}{\includegraphics{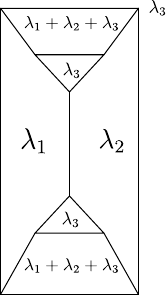}}}
\caption*{The  pair $(P_2, \lambda_2)$} 

\end{minipage}
\hfill
\begin{minipage}[h]{0.49\linewidth}

    \centering
   \scalebox{1.5}{\includegraphics[width=0.35\linewidth]{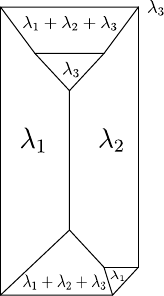}}
    \caption*{The  pair $(P_3, \lambda_3)$}

\end{minipage}

\end{figure}


\newpage
\begin{thebibliography}{99}


\bibitem{AyzGor} A.\,Ayzenberg, V.\,Gorchakov, \emph{Toric orbit spaces which are manifolds}. Arnold Math. J. (2024).
\bibitem{BBHL} A.\,Barbensi, D.\, Buck, H.\,A.\,Harrington, M.\,Lackenby,   \emph{Double branched covers of knotoids}, Communications in Analysis and Geometry, 30(5), 1007–1057 (2023). 
\bibitem{degcov} I.\,Berstein, A.L.\,Edmonds,  \emph{The degree and branch set of a branched covering}, Invent. Math. 45, 213–220 (1978).
\bibitem{orb} M.\,Boileau, B.\,Zimmermann,  \emph{The $\pi$-orbifold group of a link}, Math. Z. 200, 187–208 (1989). 

\bibitem{Bredon} G.\,E.\,Bredon, \emph{Introduction to compact transformation groups}, Pure and Applied Mathematics V.46 (1972).


\bibitem{BEMPP} V.\,M.\,Buchstaber, N.\,Yu.\,Erokhovets, M.\,Masuda, T.\,E.\,Panov, S.\,Park, \emph{Cohomological rigidity of manifolds defined by 3-dimensional polytopes},  Russian Mathematical Surveys, Volume 72, Issue 2, 199 (2017).

\bibitem{BP} V.\,M.\,Buchstaber, T.\,E.\,Panov, \emph{Toric topology}, Amer. Math. Soc., (2015).

\bibitem{CaiCho21} L.\,Cai, S.\,Choi \textit{Integral cohomology groups of real toric manifolds and small covers},  Mosc. Math. J., 21:3 , 467–492 (2021).
\bibitem{CMB} J.\,S.\,Calcut, J.\,R.\,Metcalf-Burton, \emph{Double branched covers of theta-curves}, Journal of Knot Theory and Its Ramifications 25.08 (2016).
\bibitem{RatCoh} S.\,Choi, H.\,Park, \emph{Multiplicative structure of the cohomology ring of real toric spaces}, Homology Homotopy Appl. 22, 97–115 (2020).

\bibitem{DJ} M.\,Davis, T.\,Januszkiewicz, \emph{Convex polytopes, Coxeter orbifolds and torus actions,} Duke Math. J. 62:2, 417–451 (1991).

\bibitem{DJS} M.\,Davis, T.\,Januszkiewicz, R.\,Scott, \emph{Nonpositive curvature of blow-ups}, Selecta Math. 4, 491–547 (1998).

\bibitem{Er2} N.\,Yu.\,Erokhovets, \emph{Canonical geometrization of orientable 3-manifolds defined by vector colourings of
3-polytopes}, Sb. Math., 213:6, 752–793 (2022).

\bibitem{Er} N.\,Yu.\,Erokhovets, \textit{Manifolds realized as orbit spaces of non-free $\dt^k$-actions on real moment-angle manifolds}, Proc. Steklov Inst. Math. 326, 177–218 (2024).

\bibitem{Er3} N.\,Yu.\,Erokhovets, \emph{On hyperbolic links associated to Eulerian subgraphs on right-angled hyperbolic 3-polytopes of finite volume
}, arXiv preprint arXiv:2512.03017, (2025).

\bibitem{RatSph} L.\,Ferrari, A.\,Kolpakov, A.\,W.\,Reid \textit{Infinitely many arithmetic hyperbolic rational homology 3-spheres that bound geometrically}, Trans. Amer. Math. Soc.376, no.3, 1979–1997 (2023).
\bibitem{Gor} V.\,Gorchakov, \textit{Equivariantly formal 2-torus actions of complexity one},  J. Math. Soc. Japan 77(3): 709-725 (2025).
\bibitem{Grinb} E.\,Ya.\,Grinberg, \textit{Plane homogeneous graphs of degree three without Hamiltonian cycles}, Latvian mathematical annual, Riga, Zinatne, vol. IV, 51–56  (1968).
\bibitem{Gruj} V.\,Grujić, \emph{Fundamental Groups of Three-Dimensional Small Covers}. Proc. Steklov Inst. Math. 317, 78–93 (2022).

\bibitem{EqTrian} S.\,Illman, \textit{Smooth Equivariant Triangulations of G-Manifolds for G a Finite Group}, Mathematische Annalen 233: 199-220  (1978).

\bibitem{dec} P.\,K.\,Kim and J.\,L.\,Tollefson, \emph{Splitting the PL involutions of non-prime 3-manifolds}, Michigan Math. J. 27 , 259-274 (1980).
\bibitem{Kirby} R.\,Kirby(Ed.), \emph{Problems in
Low-Dimensional Topology}, (1995)


\bibitem{Med} A.\,D.\,Mednykh, \emph{Three-dimensional hyperelliptic manifolds}. Ann. Global. Anal. Geom., 8:1, 13–19 (1990).
\bibitem{PL} E.\,Moise, \textit{Affine structures in 3-manifolds V. The triangulation theorem and Hauptvermutung}, Ann. of Math. 56 (1952).
\bibitem{MonConj} J.\,M.\,Montesinos, \emph{Surgery on links and double branched covers of $S^3$}, in: Knots, Groups and 3-Manifolds, Ann. of Math. Stud. 84, Princeton Univ. Press, Princeton, NJ, 227–260 (1975).
\bibitem{SmallCovers} H.\,Nakayama, Y.\,Nishimura, \emph{The orientability of small covers and coloring simple polytopes}, Osaka J. Math. 42(1): 243-256 (2005).



\bibitem{Unlink}  S.\,P.\,Plotnick, \textit{Finite group actions and nonseparating 2-spheres}, Proc. Amer. Math. Soc. 90, 430-432 (1984).




\bibitem{NoAct} F.\,Raymond, J.\,Tollefson, \textit{Closed 3-manifolds with no periodic maps}, Trans. Amer. Math. Soc. 221, 403-418 (1976).



\bibitem{Rolf} D.\,Rolfsen, \emph{Knots and Links}, Publish or Perish Inc., Berkeley Ca., (1976).


\bibitem{VesMed1} A.\,Yu.\,Vesnin, A.\,D.\,Mednykh, \emph{Spherical coxeter groups and hyperelliptic 3-manifolds}, Mathematical Notes, 66, 135–138 (1999).

\bibitem{VesnMed} A.\,Yu.\,Vesnin, A.\,D.\,Mednykh, \textit{Three-dimensional hyperelliptic manifolds and hamiltonian graphs}, Sib. Math. J. 40, 628–643 (1999).
\bibitem{Vesn1} A.\,Yu.\,Vesnin, \emph{Three-dimensional hyperbolic manifolds of L\"obell type}, Sib. Math. J. 28 no.~5, 731-734 (1987).
\bibitem{Vesn} A.\,Yu.\,Vesnin, \textit{Right-angled polytopes and three-dimensional hyperbolic manifolds}, Russian Math. Surveys 72, no. 2, 335–374 (2017).
\bibitem{Viro} O.\,Ya.\,Viro, \textit{Linkings, two-sheeted branched coverings and braids}, Math. USSR-Sb., 16:2, 223–236 (1972).
\bibitem{Wald} F. Waldhausen, \textit{Über Involutionen der 3-Sphäre}, Topology 8, 81-91 (1969).




\bibitem{WuYu} L.\,Wu, L.\,Yu, \emph{Haken 3-Manifolds in Small Covers}. Chin. Ann. Math. Ser. B 44, 549–560 (2023).

\end{thebibliography}
\end{document}